\documentclass[12pt]{amsart}
\usepackage{amsmath,amssymb,amsthm,amsrefs}
\usepackage{enumitem}
\usepackage{color,hyperref}
\usepackage{color}
\hypersetup{colorlinks,breaklinks,
	linkcolor=blue,urlcolor=blue,
	anchorcolor=blue,citecolor=blue}
\usepackage{cleveref}
\theoremstyle{plain}
\newtheorem{theorem}{Theorem}[section]

\newtheorem{lemma}[theorem]{Lemma}
\theoremstyle{definition}
\newtheorem{definition}[theorem]{Definition}

\newtheorem{remark}[theorem]{Remark}

\makeatletter
\@addtoreset{equation}{section}
\makeatother

\makeatletter
\newcommand{\Spvek}[2][r]{%
  \gdef\@VORNE{1}
  \left(\hskip-\arraycolsep%
    \begin{array}{#1}\vekSp@lten{#2}\end{array}%
  \hskip-\arraycolsep\right)}

\def\vekSp@lten#1{\xvekSp@lten#1;vekL@stLine;}
\def\vekL@stLine{vekL@stLine}
\def\xvekSp@lten#1;{\def\temp{#1}%
  \ifx\temp\vekL@stLine
  \else
    \ifnum\@VORNE=1\gdef\@VORNE{0}
    \else\@arraycr\fi%
    #1%
    \expandafter\xvekSp@lten
  \fi}
\makeatother

\begin{document}
\title[Spectral bounds for ergodic Jacobi operators]
{Spectral bounds for ergodic Jacobi operators\\}
\author{Burak Hat\.{i}no\u{g}lu}
\address{Department of Mathematics, Michigan State University, East Lansing MI 48829, U.S.A.}
\email{hatinogl@msu.edu}

\subjclass[2020]{47B36, 47B80, 47E05}


\keywords{ergodic Jacobi operators, spectral estimates, logarithmic capacity, Lyapunov exponent, density of states measure}

\begin{abstract}
We consider ergodic Jacobi operators and obtain estimates on the Lebesgue measure and the distance between maximum and minimum points of the spectrum in terms of the Lyapunov exponent. Our proofs are based on results from logarithmic potential theory and their connections with spectral theory of Jacobi operators. 
\end{abstract}
\maketitle

\section{Introduction}\label{Sec1}

The Jacobi operator $J$ acting on the Hilbert space $l^2(\mathbb{N})$ is the self-adjoint operator associated with the infinite Jacobi matrix
\begin{equation*}
 \begin{pmatrix}
b_1 & a_1 & 0 & 0 \\
a_1 & b_2 & a_2 & 0 \\
0  &  a_2 & b_3 & \ddots  \\
0 & 0 & \ddots & \ddots
\end{pmatrix},
\end{equation*}
given by
\begin{equation*}
    (J\psi)_n = a_{n-1}\psi_{n-1} + b_n\psi_n + a_n\psi_{n+1}, 
\end{equation*}
where $a_0 = 0$, $a_n > 0$ and $b_n \in \mathbb{R}$ for any $n \in \mathbb{N}$. Discrete Schr\"{o}dinger operators on the real line give a subclass, where $a_n = 1$ for any $n \in \mathbb{N}$. When the index set is $\mathbb{Z}$ one gets the two-sided Jacobi matrix, but for simplicity of the formulation we will stick with the index set $\mathbb{N}$. 

If $(\Omega,d\mu)$ is a probability measure space and $T : \Omega \rightarrow \Omega$ is an ergodic transformation, then one defines Jacobi parameters for $\omega \in \Omega$ and $n \in \mathbb{N}$ as
\begin{equation*}
    a_n(\omega) = f(T^n\omega) \qquad \text{and} \qquad b_n(\omega) = g(T^n\omega),
\end{equation*}
where $f$ and $g$ are measurable functions from $\Omega$ to $\mathbb{R}$ with $f$ positive, bounded and invertible, and $g$ bounded. Then $\{J_{\omega}\}_{\omega \in \Omega}$ defines an ergodic family of Jacobi operators with parameters $\{a_n(\omega),b_n(\omega)\}_{n \in \mathbb{N}}$.

Ergodic Jacobi, and specifically ergodic Schr\"{o}dinger operators appear in solid state physics and quantum Hall effect, serving various models such as crystals, disordered systems or quasi-crystals, see \cites{CFKS87,DF22} and references therein. Besides appearing in models from mathematical physics, the study of ergodic Jacobi and Schr\"{o}dinger operators have been an active research area, combining ideas from dynamical systems, analysis, potential theory and topology. Moreover, Jacobi operators with stochastic parameters show unusual spectral behavior like dense point spectrum, singular continuous spectrum and Cantor spectrum. For example, the almost Mathieu operator is one of the well-studied ergodic Jacobi operators, where the Jacobi parameters are $a_n = 1$ and $b_n(\omega) = 2\lambda\cos(2\pi(\omega + n\alpha))$ with $\alpha, \omega \in \mathbb{T}:= \mathbb{R} / \mathbb{Z}$ and $\lambda \in \mathbb{R} \setminus \{0\}$. For any $\lambda \neq 0$, the spectrum of the almost Mathieu operator is a Cantor set. This is known as the \textit{Ten Martini Problem} and was solved by Avila and Jitomirskaya \cite{AJ09}. For $|\lambda| = 1$, the spectrum of the almost Mathieu operator is purely singular continuous, which was also proved by Jitomirskaya \cite{J21}.

The spectrum of $J_{\omega}$ is nonrandom as follows: For an ergodic family $\{J_{\omega}\}_{\omega \in \Omega}$, there exists a set $\Sigma$ independent of $\omega$ such that $\sigma(J_{\omega}) = \Sigma$ for almost every $\omega$ with respect to the measure $\mu$ \cite{CFKS87}. Throughout the paper we call this set $\Sigma$ the spectrum of $J_{\omega}$. Same nonrandomness is valid for the absolutely continuous, the singular continuous and the pure point spectra \cite{CFKS87}.

In this paper we consider bounds on the Lebesgue measure and the distance between maximum and minimum points of the spectrum $\Sigma$. Similar bounds were obtained for periodic Jacobi matrices \cite{H24}. Our bounds depend on two fundamental object, the Lyapunov exponent and the density of states measure, both of which are introduced in Section \ref{Sec2}. These bounds are obtained using tools from logarithmic potential theory, which are also introduced in Section \ref{Sec2}.

The paper is organized as follows.
 
Section \ref{Sec2} includes preliminaries required for our results, namely some basics of logarithmic potential theory, and the definitions and some properties of the density of states measure and the Lyapunov exponent for ergodic Jacobi operators.

Section \ref{Sec3} includes proofs of the following results and some remarks on them.
\begin{itemize}
    \item In Theorem \ref{result1}, we consider an almost sure nonrandom upper bound for the Lebesgue measure of the spectrum in terms of the Lyapunov exponent and the off-diagonal parameters of an ergodic Jacobi operator.
    \item Theorem \ref{result2} shows an almost sure nonrandom lower bound on the difference between maximum and minimum energies in terms of the logarithmic energy of the density of states measure, or equivalently the Lyapunov exponent and the off-diagonal parameters of an ergodic Jacobi operator.
    \item In Theorem \ref{result3}, we prove that in the case that the Lyapunov exponent is H\"{o}lder continuous on the spectrum, the difference between maximum and minimum energies is almost surely bounded below by nonrandom numbers depending on the Lyapunov exponent and the off-diagonal parameters of an ergodic Jacobi operator and H\"{o}lder continuity parameters of the Lyapunov exponent.
\end{itemize}

\section{Preliminaries}\label{Sec2}

We obtain our results using estimates on logarithmic capacity, so let's recall some basics of logarithmic potential theory, which can be found e.g. in \cite{R95}.

\begin{definition}
Let $\mu$ be a finite Borel measure, supported on a compact subset of the complex plane. Then \textit{logarithmic potential} of $\mu$ is the function $U^{\mu}:\mathbb{C}\rightarrow(-\infty,\infty]$ defined by 
\begin{equation*}
U^{\mu}(z):=\int\log\frac{1}{|z-\omega|}d\mu(\omega)
\end{equation*}
\end{definition}

\begin{definition}
Let $\mu$ be a finite Borel measure, supported on a compact subset of the complex plane. Its \textit{logarithmic energy} $I(\mu)\in(-\infty,\infty]$ is defined by
\begin{equation*}
I(\mu):=\int U^{\mu}(z)d\mu(z)=\int\int\log\frac{1}{|z-\omega|}d\mu(\omega)d\mu(z)
\end{equation*}
\end{definition}

\begin{definition}
Let $K$ be a compact subset of $\mathbb{C}$ and $M(K)$ be the set of Borel probability measures compactly supported inside $K$. The measure $\mu_{K}\in M(K)$ is called the \textit{equilibrium measure} for $K$ if $$I(\mu_{K})=\displaystyle\inf_{\mu\in M(K)}I(\mu).$$
\end{definition}

Now we are ready to define the logarithmic capacity.

\begin{definition}
The \textit{logarithmic capacity} of a subset $E$ of the complex plane is given by
\begin{equation*}
Cap(E):=\sup_{\mu\in M(E)}\exp(-I(\mu)).
\end{equation*}
In particular if $K$ is compact with equilibrium measure $\mu_{K}$, then $Cap(K)=\exp(-I(\mu_{K}))$.
\end{definition}

We will use following properties of the logarithmic capacity in our results:
\begin{theorem}\label{CapacityResults}\emph{\textbf{(}\cite{R95} \textit{Theorem 5.1.2(a), Corollary 5.2.4, Theorem 5.3.1, Theorem 5.3.2(c)}\textbf{)}}
\begin{itemize}
    \item[(a)] If $K_1 \subset K_2$, then $Cap(K_1) \leq Cap(K_2)$. 
    \item[(b)] If $a\leq b$, then $Cap([a,b]) = (b-a)/4$.
    \item[(c)] Let $K$ be a compact subset of $\mathbb{C}$, and let $T: K \rightarrow \mathbb{C}$ be a map satisfying
    \begin{equation*}
        |T(z) - T(w)| \leq A|z-w|^{\alpha},
    \end{equation*}
    where $z,w \in K$, and $A$ and $\alpha$ are positive constants. Then
    \begin{equation}
        Cap(T(K)) \leq A~Cap(K)^{\alpha}.
    \end{equation}
    \item[(d)] Let $K$ be a compact subset of $\mathbb{R}$. Then
    \begin{equation}
        Cap(K) \geq |K|/4,
    \end{equation}
    where $|K|$ denotes the Lebesgue measure of $K$.
\end{itemize} 
\end{theorem}

Next, we introduce two fundamental objects related with ergodic Jacobi operators: the \textit{density of states measures} and the \textit{Lyapunov exponent}. These two objects and more on the spectral theory of ergodic Jacobi operators can be found e.g. in \cites{CFKS87,DF22,T00}. 

\begin{definition}
    Let $\{J_{\omega}\}_{\omega \in \Omega}$ be a family of ergodic Jacobi operators defined as in the Introduction section. The \textit{density of states measure} is the measure $dN$ defined by
\begin{equation}
    \int h~dN = \mathbb{E}(\langle \delta_0,h(J_{\omega})\delta_0\rangle)
\end{equation}
for bounded measurable $h$, where $\mathbb{E}(f) = \int f(\omega)~d\mu(\omega).$
\end{definition}
The density of states measure is a probability measure supported on the spectrum.

Given an ergodic family $\{J_{\omega}\}_{\omega \in \Omega}$ of Jacobi operators, the solutions of the eigenvalue problem
\begin{equation*}
   a_{n-1}(\omega)\psi_{n-1} + b_n(\omega)\psi_n + a_n(\omega)\psi_{n+1} = z\psi_n
\end{equation*}
for $n \geq 1$ and $a_0 := 0$ are represented by the transfer matrices as
\begin{equation*}
    \begin{pmatrix}
        \psi_{n+1}\\
        \psi_n
    \end{pmatrix} = A_z^n(\omega) \begin{pmatrix}
        \psi_{1}\\
        \psi_0
    \end{pmatrix},
\end{equation*}
where
\begin{equation}\label{Aznomega}
    A_z^n(\omega)  = A_z(T^{n-1}\omega) A_z(T^{n-2}\omega) \cdots A_z(T\omega) A_z(\omega)
\end{equation}
for $n \in \mathbb{N}$, $\omega \in \Omega$, $z \in \mathbb{C}$ and
\begin{equation}\label{Azomega}
    \displaystyle A_n(\lambda) := \frac{1}{a_n(\omega)}\begin{pmatrix}
\lambda - b_n(\omega) & -a_{n-1}(\omega)\\
a_n(\omega) & 0
\end{pmatrix}.
\end{equation}
Following result introduces the \textit{Lyapunov exponent}.
\begin{theorem}\label{Lyapunov exponent}\emph{\textbf{(}\cite{DF22}, \textit{Proposition 4.4.1}\textbf{)}}
For every $z \in \mathbb{C}$, there is a number $L(z) \in [0,\infty)$, called the Lyapunov exponent satisfying
\begin{align*}
    L(z) &= \inf_{n \geq 1} \frac{1}{n} \mathbb{E}\big(\log||A_z^n(\omega)||\big)\\
    &= \lim_{n \rightarrow \infty} \frac{1}{n} \mathbb{E}\big(\log||A_z^n(\omega)||\big)\\
    &= \lim_{n \rightarrow \infty} \frac{1}{n} \log||A_z^n(\omega)|| 
\end{align*}
for $\mu-a.e.$ $\omega \in \Omega$ with the notations \eqref{Aznomega} and \eqref{Azomega}.
\end{theorem}
The Lyapunov exponent is subharmonic on the complex plane. It is harmonic and positive on $\mathbb{C} \setminus \Sigma$. The famous Thouless formula relates the Lyapunov exponent and the density of states measure.
\begin{theorem}\label{Thouless}\normalfont \textbf{(Thouless Formula)} \emph{\textbf{(}\cite{S07}, \textit{Theorem 7.1(f)}\textbf{)}}
    For every $z \in \mathbb{C}$, we have
    \begin{equation}\label{Thoulessformulaequation}
        L(z) = \int \log|E-z|~dN(E) - \log(A),
    \end{equation}
    where $A := \lim_{n \rightarrow \infty}(a_1(\omega)\cdots a_n(\omega))^{1/n}$ for $\mu-a.e.$ $\omega \in \Omega$.
\end{theorem}
Thouless Formula also relates spectral theory with potential theory as the right hand side of \eqref{Thoulessformulaequation} is the negative of the logarithmic potential of the density of states measure.

\section{Results, proofs and some remarks}\label{Sec3}
In our results, the following lemma will be a main tool that represents the logarithmic capacity of the spectrum of an ergodic family of Jacobi operators in terms of the Lyapunov exponent and the equilibrium measure of the spectrum.
\begin{lemma}\label{SimonLemma}\emph{\textbf{(}\cite{S07} \textit{Equation (1.37)}\textbf{)}}
    Let $\{J_{\omega}\}_{\omega \in \Omega}$ be a family of ergodic Jacobi operators with parameters $\{a_n(\omega),b_n(\omega)\}_n$ and $L$ be its Lyapunov exponent. Then for a.e. $\omega$,
    \begin{equation}
        Cap(\Sigma) = A \exp\Big(\int L(E) d\mu_{\Sigma}(E)\Big),
    \end{equation}
    where $A := \lim_{n \rightarrow \infty}(a_1(\omega)\cdots a_n(\omega))^{1/n}$ and $\mu_{\Sigma}$ is the equilibrium measure of the spectrum.
\end{lemma}

\begin{theorem}\label{result1}
    Let $\{J_{\omega}\}_{\omega \in \Omega}$ be a family of ergodic Jacobi operators with parameters $\{a_n(\omega),b_n(\omega)\}_n$ and $L$ be its Lyapunov exponent. Then for a.e. $\omega$,
    \begin{equation}\label{result1eq}
    |\Sigma| \leq 4 A \exp\Big(\sup_{E \in \Sigma}|L(E)|\Big)
    \end{equation}
    where $A := \lim_{n \rightarrow \infty}(a_1(\omega)\cdots a_n(\omega))^{1/n}$ and $|K|$ denotes the Lebesgue measure of the set $K$.
\end{theorem}

\begin{proof}
    Using Lemma \ref{SimonLemma} and the fact that $\mu_{\Sigma}$ is a probability measure supported on $\Sigma$, we get
    \begin{equation*}
        Cap(\Sigma) \leq A \exp\Big(\sup_{E \in \Sigma}|L(E)|\Big).
    \end{equation*}
    On the other hand, $\Sigma$ is a compact subset of the real line, so by item (d) of Theorem \ref{CapacityResults} we get the desired result as
    \begin{equation*}
        \frac{|\Sigma|}{4} \leq Cap(\Sigma) \leq A \exp\Big(\sup_{E \in \Sigma}|L(E)|\Big).
    \end{equation*}
\end{proof}

\begin{remark}
    The equality is obtained for the discrete Laplacian, i.e. $a_n = 1, b_n = 0$, since $\Sigma = [-2,2]$ and $L(E) = 0$ for $E \in \Sigma$.
\end{remark}

\begin{remark}\label{examples}
    For some important classes of ergodic Jacobi operators, the Lyapunov exponent is zero on the spectrum, so inequality \eqref{result1eq} becomes $$|\Sigma| \leq 4 A.$$ Some of these classes are as follows:

    \begin{enumerate}
    \item \textit{\textbf{subcritical or critical almost Mathieu operator:}} The almost Mathieu operator is one of the best-understood examples of ergodic Schr\"{o}dinger operators, where the potential is $b_n(\omega) = 2\lambda\cos(2\pi(\omega + n\alpha)$ with frequency $\alpha$ and phase $\omega$ in $\mathbb{T}:= \mathbb{R} / \mathbb{Z}$ and coupling constant $\lambda \in \mathbb{R} \setminus \{0\}$. When $0 < |\lambda| < 1$ (subcritical) or $\lambda = \pm 1$ (critical), the Lyapunov exponent is zero everywhere on the spectrum.\\

    \item \textit{\textbf{periodic Schr\"{o}dinger operators:}} By Floquet-Bloch theory the spectrum of a periodic Schr\"{o}dinger operator has band-gap structure and purely absolutely continuous spectrum \cites{Kuc16,T00}. Since the essential closure of the zero set of the Lyapunov exponent is the ac-spectrum \cite{DF22}, Lyapunov exponents for periodic Jacobi operators are zero on the spectrum.\\

    \item \textit{\textbf{analytic quasi-periodic Schr\"{o}dinger operators without supercritical-regime:}} For analytic quasi-periodic Schr\"{o}dinger operators, the potential $b_n(\omega) = f(\alpha n + \omega)$ is generated from the real valued analytic function $f$ defined on $\mathbb{T}$ with phase $\omega \in \mathbb{T}$ and irrational frequency $\alpha \in \mathbb{T}$. Avila obtained spectral characterizations for analytic potentials using complexified Lyapunov exponents \cite{Avi15}. He considers the Lyapunov exponent with 
    \begin{equation*}
        A_z^n(\omega+i\epsilon) = \begin{pmatrix}
        z-b_n(\omega+i\epsilon) & -1\\
        1 & 0 \end{pmatrix},
    \end{equation*}
which is obtained by complexifying the phase $\omega$ using analyticity of the potential $f$. We denote this complexified Lyapunov exponent by $L_{\epsilon}(z)$. Then the spectrum with an analytic potential decomposes into three mutually disjoint sets as follows: 
\begin{itemize}
    \item The energy $E$ is \textit{subcritical} if $L_{\epsilon}(E)$ is zero in a neighborhood of $\epsilon = 0$.
    \item The energy $E$ is \textit{critical} if $L_{\epsilon}(E)$ is zero at $\epsilon = 0$, but $E$ is not subcritical.
    \item The energy $E$ is \textit{supercritical} if $L_{\epsilon}(E)$ is positive at $\epsilon = 0$.
\end{itemize}

According to this classification of energies, if the supercritical regime is empty, then the Lyapunov exponent is zero everywhere on the spectrum.\\\

\item \textit{\textbf{strictly ergodic subshift operators:}} This class of ergodic Schr\"{o}dinger operators are given by strictly ergodic subshifts on finite alphabets \cites{Dam07,S07}. It is expected that the majority of these operators have purely singular continuous spectrum as a Cantor set of zero Lebesgue measure and zero Lyapunov exponent on the spectrum \cite{Dam07}. 
\end{enumerate}
\end{remark}

\begin{theorem}\label{result2}
  Let $\{J_{\omega}\}_{\omega \in \Omega}$ be a family of ergodic Jacobi operators with Jacobi parameters $\{a_n(\omega), b_n(\omega)\}_n$. Also let $\lambda_m := \min_{E \in \Sigma}E$ and $\lambda_M := \max_{E \in \Sigma}E$, and $I(\mu)$ denote the logarithmic energy of the measure $\mu$. Then 
    \begin{equation}\label{result2corollary}
        \lambda_M - \lambda_m \geq 4 \exp(-I(dN)),
    \end{equation}
    where $dN$ is the density of states measure of $J$. Equivalently we have
    \begin{equation}\label{result2ineqLyapunov}
         \lambda_M - \lambda_m \geq 4 A \exp\Big(\int L(E) ~dN(E)\Big),
    \end{equation}
    where $A := \lim_{n \rightarrow \infty}(a_1(\omega)\cdots a_n(\omega))^{1/n}$ for a.e. $\omega$ and $L$ is the Lyapunov exponent of $J_{\omega}$.
\end{theorem}

\begin{proof}
    Using Lemma \ref{SimonLemma} and the Thouless formula (Theorem \ref{Thouless}) and recalling that $U^{\mu}$ denotes the logarithmic potential of a measure $\mu$ we get
    \begin{align}
        Cap(\Sigma) &= A\exp\Big(\int L(E)~d\mu_{\Sigma}(E)\Big) \\
                    &= A\exp\Big(\int -U^{dN}(E)~d\mu_{\Sigma}(E) - \int\log(A)~d\mu_{\Sigma}(E)\Big) \\
                    &= A\exp\Big(\int -U^{\mu_{\Sigma}}(E)~dN(E) - \log(A)\Big) \\
                    &\geq \exp\Big(\int -U^{dN}(E)~dN(E)\Big) \label{ineq} \\
                    &= \exp(-I(dN)) \label{logEnergy}.
    \end{align}
    The inequality above follows from the definition of the equilibrium measure and the fact that density of states measure $dN$ is a probability measure supported on $\Sigma$. 
    
    Note that the logarithmic energy of the density of states measure $I(dN)$ is given in terms of the Lyapunov exponent by Thouless formula as
    \begin{equation}
        I(dN) = - \int L(E) ~dN(E) - \log A,
    \end{equation}
    so we can replace \eqref{logEnergy} by 
    \begin{equation*}
         A \exp \Big(\int L(E)dN(E)\Big).
    \end{equation*}

    On the other hand logarithmic capacity is monotone and the logarithmic capacity of an interval is the one fourth of the size of that interval (Theorem \ref{CapacityResults} (a),(b)), so we have
    \begin{equation}\label{CapMaxMinEstimate}
        Cap(\Sigma) \leq Cap([\lambda_m,\lambda_M]) = (\lambda_M - \lambda_m)/4
    \end{equation}
    and hence get the desired results.
\end{proof}

\begin{remark}
    The equality for \eqref{result2ineqLyapunov} is obtained for the discrete Laplacian, i.e. $a_n = 1, b_n = 0$, since $\Sigma = [-2,2]$ and $L(E) = 0$ for $E \in \Sigma$.
\end{remark}

\begin{remark}
    If the Lyapunov exponent is zero everywhere on the spectrum, specifically for the classes we discussed in Remark \ref{examples} we have 
    \begin{equation}
        \lambda_M - \lambda_m \geq 4 A.
    \end{equation}
\end{remark}

\begin{remark}
    If the Lyapunov exponent is zero everywhere on $\Sigma$, then we showed that
    \begin{equation}
        |\Sigma| \leq 4A \leq \lambda_M - \lambda_m.
    \end{equation}
    Therefore if the spectrum is a continuum set in this case, then 
    \begin{equation}
        \lambda_M - \lambda_m = |\Sigma| = 4A.
    \end{equation}
\end{remark}

If the Lyapunov exponent is H\"{o}lder continuous, then we obtain another lower estimate for the logarithmic capacity of the spectrum and hence for the difference between maximum and minimum energies in terms of the Lyapunov exponent and H\"{o}lder continuity parameters.

\begin{theorem}\label{result3}
     Let $\{J_{\omega}\}_{\omega \in \Omega}$ be a family of ergodic Jacobi operators with parameters $\{a_n(\omega),b_n(\omega)\}_n$ and $L$ be its Lyapunov exponent. Also let $\lambda_m := \min_{E \in \Sigma}E$ and $\lambda_M := \max_{E \in \Sigma}E$. If $L$ is $\alpha$-H\"{o}lder continuous on the spectrum $\Sigma$ for $0 < \alpha < 1$, i.e.
    \begin{equation}
        |L(z) - L(w)| \leq C|z-w|^{\alpha}
    \end{equation}
    for $z,w \in \Sigma$ and $C > 0$, then
    \begin{equation}\label{MinMaxEstimateHolder}
         \lambda_M - \lambda_m \geq \Big(\frac{|L(\Sigma)|}{C4^{1-\alpha}}\Big)^{1/\alpha}
    \end{equation}
    where $L(\Sigma) := \{L(E) ~|~ E \in \Sigma\}$ and $|K|$ denotes the Lebesgue measure of a set $K$. 
\end{theorem}

\begin{proof}
    Using items (a),(b),(c) of Theorem \ref{CapacityResults}, our assumption and the estimate \eqref{CapMaxMinEstimate} we get
    \begin{equation}
        Cap(L(\Sigma)) \leq C (Cap(\Sigma))^{\alpha} \leq C \Big(\frac{\lambda_M - \lambda_m}{4}\Big)^{\alpha}.
    \end{equation}
    On the other hand, the Lyapunov exponent is real-valued and we assumed it to be continuous on the spectrum, so $L(\Sigma)$ is a compact subset of the real line. Therefore by item (d) of Theorem \ref{CapacityResults} we get
    \begin{equation}
        |L(\Sigma)|/4 \leq  Cap(L(\Sigma)) 
    \end{equation}
    and hence the desired result.
\end{proof}

\begin{remark}
    It is known that $Cap(\Sigma) \geq 1$ \cite{DF22}, so \eqref{MinMaxEstimateHolder} is non-trivial if $|L(\Sigma)| \geq 4C$.
\end{remark}

\begin{remark}
    If the image of the spectrum under the Lyapunov exponent is a continuum set, then the measure of $L(\Sigma)$ is the difference between maximum and minimum values of $L$ over $\Sigma$. Therefore inequality \eqref{MinMaxEstimateHolder} becomes
    \begin{equation}
       \lambda_M - \lambda_m \geq \Bigg(\frac{\displaystyle\sup_{E \in \Sigma}|L(E)| - \inf_{E \in \Sigma}|L(E)|}{C4^{1-\alpha}}\Bigg)^{1/\alpha}. 
    \end{equation}
    Moreover if the Lyapunov exponent is zero at some energy, e.g. if the ac-spectrum is non-empty, then from the non-negativity of the Lyapunov exponent, inequality \eqref{MinMaxEstimateHolder} becomes
    \begin{equation}
       \lambda_M - \lambda_m \geq \Bigg(\frac{\sup_{E \in \Sigma}|L(E)|}{C4^{1-\alpha}}\Bigg)^{1/\alpha}.
    \end{equation}
\end{remark}

\bibliographystyle{abbrv}
\bibliography{references}

\end{document}